\documentclass{amsart}
\pdfoutput=1

\usepackage{amsmath,amsthm,amssymb,amsfonts}
\usepackage{mathtools}
\usepackage[all,arc]{xy}
\usepackage[shortlabels]{enumitem}
\usepackage{mathrsfs}
\usepackage{tikz-cd}
\usepackage[backend=biber, style=numeric]{biblatex}
\usepackage{preamble}

\newtheorem{thm}{Theorem}[section]
\newtheorem{thmx}{Theorem}

\newtheorem{prp}[thm]{Proposition}
\newtheorem{lem}[thm]{Lemma}

\newtheorem{quest}{Question}

\theoremstyle{definition}
\newtheorem{defn}[thm]{Definition}

\theoremstyle{remark}
\newtheorem{rmk}[thm]{Remark}

\makeatletter
\let\c@equation\c@thm
\makeatother
\numberwithin{equation}{section}

\setitemize{noitemsep,topsep=0pt,parsep=0pt,partopsep=0pt}

\title{Complex Orientations are Partial Strictifications of the Unit}

\author{Doron Grossman-Naples}

\date{}

\begin{document}

\begin{abstract}

We give a higher-algebraic interpretation of complex orientations of ring spectra as ``\(\E_2\) strictifications'' of the identity element. We show that higher strictifications do not exist for most ring spectra of interest in chromatic homotopy theory.

\end{abstract}

\maketitle

\tableofcontents

\section{Introduction}
Let \(R\) be a ring spectrum. A \emph{complex orientation} of \(R\) is an element \(c\in R^2(\CPinf)\) whose restriction along the map \(\C\bP^1\xhookrightarrow{}\CPinf\) corresponds to \(1\) under the natural identification \(R^2(\C\bP^1)\cong\pi_0(R)\). Geometrically, this can be regarded as a generalized Chern class: it is a degree \(2\) \(R\)-valued characteristic class for complex line bundles, and the restriction condition calibrates it to behave well under tensor product (which is necessary for it to induce a formal group law).

As chromatic homotopy theory has developed, however, it has become clear that it is not only a computational tool for understanding geometric objects, but also a sort of analogue of algebraic number theory for Brave New Algebra. This motif appears in Balmer's interpetation of the Thick Subcategory Theorem (\cite{balmer2020},\cite{barthel2020}), for example, and Lurie's (\cite{lurie2018}) and Gregoric's (\cite{gregoric2021}) work on the moduli stacks of oriented elliptic curves and formal groups (respectively), as well as more speculative work such as Salch's ``topological Langlands program'' (\cite{salch2023a},\cite{salch2023}). It would be highly desirable to have an interpretation of complex orientations purely in terms of this higher algebra, but one is presently lacking.

The purpose of this note is to provide such an interpretation. This is achieved using what we call ``strict elements'', inspired by Hopkins's notion of ``strict units''. Our main theorem is as follows:
\begin{thmx} \label{thm:A}
Let \(R\) be a ring spectrum. Then complex orientations of \(R\) are naturally equivalent to \(\E_2\)-strictifications of the element \(1\in\pi_0 R\).
\end{thmx}

The main idea of the theorem is that the map \(\C\bP^1\to\CPinf\) is a sort of transposed Hurewicz map, using the equivalences \(\C\bP^1\cong S^2\) and \(\CPinf\simeq K(\Z,2)\). Explicitly, this arises from the fact that \(\CPinf\) is the free strict abelian group on the pointed space \(S^2\) (\cite{lurie2018}). This allows us to identify the map defining complex orientations with a certain map defined purely in terms of \(\E_2\)- and \(\E_1\)-structures on the relevant spaces.

A natural question to ask is what it looks like if we replace \(\E_2\) with \(\E_n\) for \(3\leq n\leq\infty\). Our other main theorem shows that, in most cases of chromatic interest, this is vacuous.
\begin{thmx} \label{thm:B}
If a ring spectrum \(R\) has a higher complex orientation, then \(L_{K(n)}R\simeq 0\) at all primes \(p\) for each \(0<n<\infty\).
\end{thmx}
That is to say, a ring spectrum with a higher complex orientation is supported at chromatic heights \(0\) and \(\infty\).\footnote{While this theorem is stated in terms of chromatic support, it is equivalent to describe it in terms of telescopic support. This is because, by a theorem of Hovey (\cite{hovey1995}), \(T(n)\)- and \(K(n)\)-localization are equivalent for complex-oriented ring spectra.} This theorem is proven using Ravenel and Wilson's computation of the \(K(n)\)-cohomology of Eilenberg-MacLane spaces (\cite{ravenel1980}).

\begin{rmk}These results help to explain the central role played by even-periodicity in chromatic homotopy theory. An \(\Einf\)-ring \(R\) is complex-periodic if and only if there is an oriented formal group over \(R\), and this is ultimately the case because \(\CPinf\) is the free strict abelian group on \(S^2\). Using this relation shows that the Bott map of an orientation (see \cite{lurie2018}) will factor through the map \(S^2\to\CPinf\) to give the lifting in the definition of a complex orientation.

The ``\(1\)-periodic'' version of the story is trivial because \(S^1\simeq K(\Z,1)\); the \(2\)-periodic version can be described in terms of \(\E_2\)-strictifications by Theorem A; and the \(n\)-periodic version for \(n\geq 3\) (which would involve factoring through the map \(S^n\to K(\Z,n)\)) will not occur in cases of chromatic interest by Theorem B. Thus the strictification process allowing us to describe the behavior of these rings in terms of formal groups only works in degree \(2\).
\end{rmk}

\subsection{Conventions}
We adhere to the following notation and terminology throughout this paper.

\begin{itemize}[leftmargin=*]
\item We use the term ``ring spectrum'' to mean ``homotopy ring spectrum'', and specify ``\(\Einf\)-ring'' when we wish the spectrum to have highly structured multiplication.
\item We denote by \(\Omega^{\infty-n}X\) the \(n\)th space in an \(\Omega\)-spectrum model for a spectrum \(X\), and similarly we denote by \(\Sigma^{\infty-n}X\) the spectrum \(\Sigma^{-n}\Sigma^{\infty}X\) for any pointed space \(X\).
\item We work with \(\infty\)-categories throughout, and all constructions are implicitly taken to be their \(\infty\)-categorical versions (e.g. ``pullback'' means ``homotopy pullback'').
\item We write \(\Map_{\E_n}(X,Y)\) to denote the mapping space in the category of (grouplike) \(\E_n\)-spaces and \(\Map_*(X,Y)\) to denote the mapping space in the category of pointed spaces.
\item We refer to the sphere spectrum as \(\bbS\).
\end{itemize}

\subsection{Acknowledgments}
I would like to thank my advisor Charles Rezk for helpful conversations on this topic, and in particular for suggesting the method used to prove Theorem \ref{thm:B}. Thanks are also due to Florian Reidel, who helped me resolve an error regarding corepresentability of strict elements, and to Robert Burklund, who pointed out that Theorem \ref{thm:B} could be strengthened to apply to all ring spectra.

\section{Strict Elements}
In classical algebra, we have the pleasant coincidence that the free group on one element is also the free abelian group on one element. This makes it easy to say what an element of an abelian group is: a map out of \(\Z\). In the homotopical context, though, things are a little more complicated. The free \(\infty\)-group on one element, \(\Z\), does indeed support an \(\Einf\) structure; but since this is a structure rather than a property, it is not the \emph{free} \(\Einf\) structure.\footnote{Moreover, confusingly, the free \(\E_n\)-group on one element (\(1<n<\infty\)) does \emph{not} admit an \(\Einf\) structure, since \(\Omega^n S^n\) is not an infinite loop space.}

We are therefore lead to ask: what is an ``element'' of a spectrum \(X\)? When \(X\) is a space, we might reasonably use this to refer an element of \(\pi_0 X\). In the spectral case, however, this definition is problematic. Part of the problem is philosophical: this fails to take into account the structure of a spectrum beyond its infinite loop space. In particular, when \(X\) is a connective spectrum identified with a grouplike \(\Einf\)-space, an element in this sense completely ignores the \(\Einf\) structure. However, this notion of element also has practical issues, as described in the following motivating example (c.f. the first chapter of \cite{lurie2018}).
\subsection*{Example}
Let \(R\) be an \(\Einf\)-ring, and suppose we wish to talk about the ``units'' of \(R\). One natural way to do this is by taking the pullback
\[\begin{tikzcd}
GL_1(R) \ar[d] \ar[r] & \Omega^{\infty}R \ar[d]\\
\pi_0(R)^{\times} \ar[r] & \pi_0(R).
\end{tikzcd}\]
The space \(GL_1(R)\), which is called the ``space of units'', carries an \(\Einf\)-structure inherited from the multiplication on \(R\); and, moreover, it is grouplike due to our restriction to the units of \(\pi_0 R\). It therefore lifts to a spectrum \(gl_1(R)\), the ``spectrum of units''. The functor \(gl_1:\CAlg\to\Sp\) is corepresentable by \(\bbS\{t^{\pm}\}\), the free \(\Einf\)-ring on a single invertible element.

We have a problem here: \(\bbS\{t^{\pm}\}\) is not flat over the sphere spectrum! Moreover, its homotopy groups are terribly unpleasant. For both of these reasons, its completion is not a formal group in the sense of Lurie, which limits its usefulness in chromatic homotopy theory. Even though it is the most literal version of the units of \(R\), it fails to behave like it should. This was the motivation behind Hopkins's definition of the \emph{strict} multiplicative group, \(\GG_m(R)\vcentcolon=\tau_{\geq 0}\Map_{\Sp}(H\Z,gl_1(R))\). The functor \(\GG_m\) is corepresented by the \(\Einf\)-ring of Laurent polynomials in one variable, \(\bbS[t^{\pm}]=\Sigma_+^{\infty}\Z\), which is flat over \(\bbS\) and has nice homotopy groups \(\pi_*\bbS[t^{\pm}])=(\pi_*\bbS)[t^{\pm}]\).

It is in this spirit that we make the following definition.
\begin{defn}
A \emph{strict element} of a spectrum \(X\) is an element of \(\pi_0\Map_{\Sp}(H\Z,X)\). We call \(\Map_{\Sp}(H\Z,X)\) the \emph{space of strict elements} of \(X\). 
\end{defn}
For the sake of clarity, we will refer to elements of \(\pi_0 X\) as \emph{weak elements}, and to \(\Omega^{\infty} X=\Map_{\Sp}(\bbS,X)\) as the \emph{space of weak elements}.

\begin{rmk}
The evident forgetful map from strict elements to weak elements is corepresented by the unit map \(\bbS\to H\Z\), which implies in particular that weak and strict elements coincide for rational spectra.
\end{rmk}

We can go further with this. Observe that we have \(\Map_{\Sp}(H\Z,X)\simeq\Map_{\Einf}(\Z,\Omega^{\infty}X)\) and \(\pi_0 X\simeq\Map_{\E_1}(\Z,\Omega^{\infty}X)\). It is reasonable to consider notions of ``partially strict'' elements that interpolate between these:

\begin{defn}
An \emph{\(\E_n\)-element} of a spectrum \(X\) is an \(\E_n\)-map \(\Z\to\Omega^{\infty} X\), and the \emph{space of \(\E_n\)-elements} is \(\Map_{\E_n}(\Z,\Omega^{\infty} X)\).
\end{defn}
Note that this recovers weak elements for \(n=1\) and strict elements for \(n=\infty\). The promised interpolation can be represented by the following towers of mapping spaces.

\[\begin{tikzcd}
{} & {} & \Map_{\Einf}(\Z,\Omega^{\infty}X) \ar[d]\\
\vdots \ar[r, phantom,"\simeq"] \ar[d] & \vdots \ar[r, phantom,"\simeq"] \ar[d] & \vdots \ar[d]\\
\Map_{\E_1}(\C\bP^{\infty},\Omega^{\infty-2}X) \ar[r, phantom, "\simeq"] \ar[d] & \Map_{\E_2}(S^1,\Omega^{\infty-1}X) \ar[r, phantom, "\simeq"] \ar[d] & \Map_{\E_3}(\Z,\Omega^{\infty}X) \ar[d]\\
\Map_*(\C\bP^{\infty},\Omega^{\infty-2}X) \ar[r, phantom, "\simeq"] & \Map_{\E_1}(S^1,\Omega^{\infty-1}X) \ar[r, phantom, "\simeq"] \ar[d] & \Map_{\E_2}(\Z,\Omega^{\infty}X) \ar[d]\\
{} & \Map_*(S^1,\Omega^{\infty-1}X) \ar[r,phantom,"\simeq"] & \Map_{\E_1}(\Z,\Omega^{\infty}X).
\end{tikzcd}\]

The left and center towers are obtained from the right via the loop space recognition principle. Comparing towers, we see that

\begin{enumerate}[i)]
\item \(\pi_0\Map_{\E_1}(\Z,\Omega^{\infty}X)\cong\pi_0 X\),
\item \(\pi_0\Map_{\E_2}(\Z,\Omega^{\infty}X)\cong X^2(\C\bP^{\infty})\), and
\item The forgetful functor \(E_2(\Spaces_*)^{gp}\to E_1(\Spaces_*)^{gp}\) induces a map \(f_{2\to1}:X^2(\C\bP^{\infty})\to\pi_0 X\).
\end{enumerate}

The key insight behind our main theorem is that the identifications \(\C\bP^1\cong S^2\) and \(\CPinf\simeq K(\Z,2)\) allow us to give a purely algebraic description of the restriction map in cohomology.

\begin{lem}
The map \(f_{2\to 1}:X^2(\C\bP^{\infty})\to\pi_0 X\) coincides with the map \(X^2(\CPinf)\to X^2(\C\bP^1)\cong\pi_0 X\) induced by the standard inclusion \(\C\bP^1\xhookrightarrow{}\CPinf\).
\end{lem}
\begin{proof}
The inclusion \(S^2\cong\C\bP^1\xhookrightarrow{}\C\bP^{\infty}\simeq K(\Z,2)\) can be written as \(\Sigma S^1\to BS^1\). By suspension-loop adjunction, this is the transpose of a map \(S^1\to\Omega BS^1\simeq S^1\), and in fact its transpose is the identity (since it induces the identity on \(H_1\)). Thus our restriction map is given by upper horizontal arrow in
\[\begin{tikzcd}
\pi_0\Map_*(BS^1,\Omega^{\infty-2}X) \ar[r] \ar[d,phantom, sloped, "\cong"] & \pi_0\Map_*(\Sigma S^1,\Omega^{\infty-2}X) \ar[d, phantom, sloped, "\cong"]\\
\pi_0\Map_{\E_1}(S^1,\Omega^{\infty-1}X) \ar[r] & \pi_0\Map_*(S^1,\Omega^{\infty-1}X),
\end{tikzcd}\]
and, by the lower right square in the diagram of towers above, the lower horizontal arrow is \(f_{2\to 1}\).
\end{proof}
\begin{proof}[Proof of Theorem \ref{thm:A}]
A complex orientation is a lift of \(1\in\pi_0 R\) through the map induced by the inclusion \(\C\bP^1\xhookrightarrow{}\CPinf\). The previous result implies that this is the same as a lift through \(f_{2\to 1}\), i.e. a lift from an \(\E_1\)-element to an \(\E_2\)-element.
\end{proof}

\section{Higher Complex Orientations}
Knowing now that complex orientations are \(\E_2\) lifts of \(1\), a natural question to ask is what \(\E_n\) lifts of \(1\) look like for \(n\geq 3\). These ``higher complex orientations'' turn out to be quite scarce, however. For the case \(n=\infty\), an argument of Nardin and Peterson (\cite{nardin2018}) gives us an obstruction.

\begin{prp}
Let \(R\) be a ring spectrum with \(\pi_0 R\cong\Z\). If \(1\) is a strict element of \(R\), then \(\tau_{\geq 0}R\) retracts onto \(H\Z\) and the first \(k\)-invariant of \(\Omega^{\infty} R\) is trivial.
\end{prp}
\begin{proof}
By assumption, we have a map \(H\Z\to R\) which is the identity on \(\pi_0\), and we can factor through the connective cover to get a map \(H\Z\to\tau_{\geq 0} R\). On the other hand, since \(\pi_0 R=\Z\), we have a ``cotruncation map'' \(\tau_{\geq 0} R\to H\Z\) which is also the identity on \(\pi_0\). The composition of these is an endomorphism of \(H\Z\) which is the identity on \(\pi_0\), and thus it is the identity. This means that the first fibration in the Postnikov tower for \(\Omega^{\infty}R\) splits, so its first \(k\)-invariant is trivial.
\end{proof}

Theorem \ref{thm:B} gives us an even more powerful no-go result, showing that essentially any ring of chromatic interest cannot have an \(\E_n\)-orientation for any \(n>2\).
\begin{proof}[Proof of Theorem \ref{thm:B}]
We wish to show that \(R\) is chromatically supported at heights \(0\) and \(\infty\), i.e. that \(L_{K(n)}R\simeq 0\) for all finite \(n>0\) at each prime \(p\). We may assume without loss of generality that \(R\) is \(p\)-local, and work in \(\Sp_{(p)}\).

We have a tower \[\bbS\xrightarrow{\simeq}\Sigma^{\infty-1}K(\Z,1)\to\Sigma^{\infty-2}K(\Z,2)\to\Sigma^{\infty-3}K(\Z,3)\to\dotsb\] such that an \(\E_n\)-orientation on \(R\) corresponds to a factorization of the unit \(\bbS\to R\) through the \(n\)th term of the tower.\footnote{Incidentally, this means that the free \(\Einf\)-ring on the \((n+1)\)st term of the tower is the universal example of an \(\Einf\)-ring with \(\E_n\)-orientation. By the Theorem, these rings must all have chromatic support concentrated in \(\{0,\infty\}\).} Therefore, we need to show that if the unit map of \(R\) factors through \(\Sigma^{\infty-3}K(\Z,3)\), \(R\) is \(K(n)\)-acyclic.

It is enough to show that the unit map of \(K(n)\otimes R\) is zero, since a ring with \(1=0\) is trivial. By Ravenel-Wilson's computation of the Morava K-theory of Eilenberg-MacLane spaces (Theorem 12.1 of \cite{ravenel1980}), we find that \(K(n)_*K(\Z,3)\) is concentrated in even degree. It follows that the map \(\bbS\to K(n)\otimes\Sigma^{\infty-3}K(\Z,3)\) must be zero, and so, therefore, is the unit of \(K(n)\otimes R\).
\end{proof}
We conclude with two questions that warrant further investigation.

\begin{quest}
Can this theory be generalized to \(G\)-orientations for more general groups \(G\)?
\end{quest}
There is an evident mod \(2\) analogue of this theory for \(MO\)-orientations, where we replace the essential equivalences \(\C\bP^1\cong S^2\) and \(\CPinf\simeq K(\Z,2)\) with \(\R\bP^1\cong S^1\) and \(\R\bP^{\infty}\simeq K(\Z/2,2)\). Both of these cases work because an orientation is determined by a single characteristic class (the first generalized Chern or Stiefel-Whitney class), which is a consequence of the splitting principle. The same is not true for other groups of interest, such as \(Spin\), where (writing \(\mathbb{T}\) for the maximal torus) \(H^*(BSpin(n))\to H^*(B\mathbb{T}_{Spin(n)})\) is not fully faithful or even injective for \(n\geq 8\). However, there is a generalized splitting principle for such groups, which may permit a ``twisted'' version of this story.

\begin{quest}
What is the obstruction theory for liftng an \(\E_n\)-element of a spectrum \(X\) to an \(\E_m\)-element for \(m>n\)?
\end{quest}
For the case \(n=1\), \(m=2\), this will evidently coincide with the cellular obstruction theory for extending from \(\C\bP^1\to\CPinf\). The case for greater \(n\) and \(m\) is less clear, however. By Dunn additivity (\cite{lurie2017}), we have a canonical equivalence \(\E_{n+1}(\Spaces_*)\simeq\E_1(\E_n(\Spaces_*))\), so it should be possible to give an inductive description using \(\E_1\) obstruction theory.

\printbibliography

\end{document}